\documentclass[12pt]{amsart} 

\usepackage[margin=2.9cm]{geometry}             
\usepackage{color,ulem}
\usepackage{graphicx}
\usepackage{amssymb, enumerate,bm}
\usepackage[matrix,arrow,ps]{xy}
\usepackage{epstopdf, amsmath}
\usepackage{paralist}
\newcommand{\C}{\mathbb C}

\newcommand{\R}{\mathbb R}

\newcommand{\transp}{\,^t}

\newtheorem{theo}{Theorem}[section]
\newtheorem{lemma}[theo]{Lemma}
\newtheorem{cor}[theo]{Corollary}

\theoremstyle{remark}
\newtheorem{remark}[theo]{Remark}
\theoremstyle{example}

\theoremstyle{definition}
\newtheorem{defi}[theo]{Definition}
\numberwithin{equation}{section}

\begin{document}

\begin{abstract} The existence of a nondefective  stationary disc  attached to a nondegenerate  model quadric in  $\C^N$  is a necessary condition  to ensure the unique $1$-jet determination of the   lifts of  a  key family of  stationary discs \cite{be-me2}.
 In this paper,  we  give an elementary proof of the equivalence  when the   model quadric  is strongly pseudoconvex, recovering a result of Tumanov  \cite{tu}. Our proof is based on the explicit expression of  stationary discs, and opens up a conjecture for  the unique $1$-jet determination  to hold when the model is not necessarily  strongly pseudoconvex.

\end{abstract}

\author{Florian Bertrand  and Francine Meylan}
\title[The $1$-jet determination  of  stationary discs]{The $1$-jet determination  of  stationary discs attached to  generic CR submanifolds 
}
\thanks{Research of the first  author was  supported by the Center for Advanced Mathematical Sciences and by an URB grant from the American University of Beirut.}



\maketitle


\section*{Introduction} 
In the  paper \cite{be-me2},  we  gave an explicit construction of  stationary discs attached to  a strongly Levi nondegenerate model quadric,
and obtained a necessary condition, namely the existence of a nondefective  stationary disc,   to ensure the unique $1$-jet determination of their lifts. As emphasized in  \cite{be-me}, this is  a crucial step to deduce the $2$-jet determination of CR automorphisms of  strongly Levi nondegenerate CR submanifolds. Tumanov proved in \cite{tu}  that the existence of a nondefective stationary disc   is also a sufficient  condition if the model is strongly pseudoconvex. His proof is based on the fact that, in the strongly pseudoconvex case,   lifts  of stationary discs can be made 
"supporting" in a suitable  (nonlinear) system of coordinates (Proposition 3.9  in \cite{tu}, see also Proposition 2.2' in \cite{sc-tu}).
		
One purpose of the present  paper is   to  address   the question of  recovering  Tumanov's result using  the  explicit construction of stationary discs, without performing a nonlinear change of coordinates that would not preserve the model quadric (Theorem \ref{propder}). The elementary proof  we give  reveals a possible generalization to merely  strongly Levi nondegenerate  model quadric of this unique  $1$-jet determination (Subsection 4.1 and its open questions). 

Another purpose of this paper   is to  discuss the link between various definitions related to the model quadric (see Lemma \ref{nil0}),  and also between different technics  used to obtain  finite jet determination of biholomorphisms preserving CR  submanifolds.  We  thus generalize, for a class of Levi generating submanifolds,  the $2$-jet determination of biholomorphisms preserving strongly Levi  nondegenerate  CR  submanifolds obtained in \cite{be-bl-me}  to  merely Levi   nondegenerate CR submanifolds (see Theorem \ref{nil1}). We recall that   the $2$-jet determination of biholomorphisms preserving  (Levi generating)  Levi  nondegenerate  CR  real-analytic submanifolds does not hold in general, due to the counterexamples obtained in \cite{gr-me}. We refer to the surveys of Zaitsev \cite{za} and, Lamel and Mir \cite{la-mi3}, and references therein, for an overview of  finite jet determination problems.

\section{Hermitian matrices associated to the Levi map of  generic  submanifolds}

In this section, we recall various definitions related to the Levi map (see p. 41 \cite{BER}) of a generic submanifold via its associated Hermitian matrices. We also  recall recent results about $2$-jet determination of CR mappings in higher codimension, and obtain a new result in this direction. 

\vspace{0.5cm}

Let $M \subset \C^{n+d}$ be a  $\mathcal{C}^{4}$ generic real submanifold of real codimension $d\ge 1$ through $p=0$ given locally by the following system of equations 
\begin{equation}\label{eqred0}
\begin{cases}
 \Re e  w_1= \transp\bar z A_1 z+ O(3)\\
\ \ \ \ \vdots \\
\Re e  w_d = \transp\bar z A_d z+ O(3)
\end{cases}
\end{equation}
where $A_1,\hdots,A_d$ are Hermitian matrices of size $n$. In the remainder O(3), the variables $z$ and $\Im m w$  are respectively of weight one and two. 
We associate to $M$ the model quadric $M_H$ given by
\begin{equation}\label{eqred1}
\begin{cases} \Re e  w_1 = \transp\bar z A_1 z\\
\ \ \ \ \vdots \\
\Re e  w_d = \transp\bar z A_d z.
\end{cases}
\end{equation}

  Recall that $M$ is of finite type at $0$  with $2$ the only  H\"ormander number  if and only  $M$ is Levi generating at $0$ if and only if the matrices $A_1,\ldots,A_d$ are linearly independent (see \cite{bl-me} for instance). We say that $M$ is {\it Levi nondegenerate at $0$} 
  in case $\cap_{j=1}^d\mathrm{Ker}A_j=\{0\}$.   The submanifold $M$ is {\it strongly Levi nondegenerate at $0$} (resp. {\it strongly pseudoconvex at $0$}) if there exists 
$b \in \Bbb R^d$ such that the matrix $\sum_{j=1}^d b_jA_j$ 
is invertible (resp. positive definite). Note that in case $M$ is strongly pseudoconvex,  the matrix  $\sum_{j=1}^d b_jA_j$ may be chosen to 
be the identity matrix after a linear holomorphic change of variables. We also recall the following definition from \cite{be-me}.
\begin{defi}\label{definondeg}
  Let $M$ be  strongly Levi nondegenerate at $0$ given by \eqref{eqred0} and let $b\in \R^d$ be such that $\sum_{j=1}^d b_jA_j$ is invertible. 
We say that  $M$ is {\it $\mathfrak{D}$-nondegenerate} at 0 if there exists $V \in \C^{n}$ such that, if we set $D_0$ to be the $n \times d$ matrix whose $j^{th}$ column is
 $A_jV$, then $\Re e (\transp \overline D_0 (\sum_{j=1}^d b_jA_j)^{-1}D_0)$ is invertible.
\end{defi}
The following lemma was pointed out to us by Bernhard Lamel; we thank him for that. This lemma shows the connection between  the Segre sets introduced by Baouendi, Ebenfelt and Rothschild (see Chapter $10$ in \cite{BER})  and the matrix    $D_0$  introduced  in  
the previous definition. 
\begin{lemma}\label{nil0}
 The model quadric   $M_H$  given by \eqref{eqred1} is of finite type with    Segre number  $2$ at $0$ if and only if there exists a $V$ such that the $n \times d$ matrix whose $j^{th}$ column is
 $A_jV$ has rank $d.$  
\end{lemma}
\begin{proof}       
Let  $S_2(0)$ be the Segre set of order $2$ at $0.$  A direct computation yields to
\begin{equation*}
 S_2(0)= \{ (z,w) \in \Bbb C^n \times \Bbb C^d \ |\ \exists \zeta \in \Bbb C^n,  \ \ w_j={^t\zeta}A_jz\}.
\end{equation*}
 Let $d_2(0)$ be the generic rank of the map 
\begin{equation} \label{eqpr}
 pr: (z,\zeta) \in \Bbb C^n \times \Bbb C^n \longrightarrow (z, {^t\zeta}A_1z, \dots,  {^t\zeta}A_dz) \in \Bbb C^{n+d} .
\end{equation}
According to Theorem 10.5.5 in \cite{BER}, the quadric  $M_H$ is of finite type at $0$ with Segre number $2$ at $0$ if and only if $d_2(0) = n+d.$
According to the form of the map $pr$ given by  \eqref{eqpr}, we obtain that  $d_2(0) = n+d$ if and only if 
the $n \times d$ matrix whose $j^{th}$ column is
 $A_jV$ has rank $d$ for some $V.$ This achieves the proof.
\end{proof}

\vspace{0.5cm}

We now highlight two recent results on  $2$-jet determination of CR automorphisms. Tumanov obtained in \cite{tu3}  the following theorem for strongly pseudoconvex  submanifolds.
\begin{theo} \label{strictly}\cite{tu3}
Let $M \subset \C^N$ be a  $\mathcal{C}^4$ generic  real submanifold  given by \eqref{eqred0}.  Assume that $M$ is strongly pseudoconvex and that the matrices 
$A_1,\ldots,A_d$ are linearly independent. Then any germ at $0$ of CR automorphism  of $M$ of class $\mathcal{C}^3$  is uniquely determined by its $2$-jet at $0.$
\end{theo} 
In \cite{be-me}, we  proved that  $2$-jet determination holds under the $\mathfrak{D}$-nondegeneracy of the submanifold (see also \cite{be-bl-me}). More precisely,
\begin{theo}\cite{be-me}
Let $M \subset \C^N$ be a  $\mathcal{C}^4$ generic  real submanifold given by \eqref{eqred0}. Assume that $M$ is {\it $\mathfrak{D}$-nondegenerate} at $0$.  Then any germ at $0$ of CR automorphism  of $M$ of class $\mathcal{C}^3$  is uniquely determined by its $2$-jet at $0.$
\end{theo}
 In light of Lemma \ref{nil0}, we may rewrite  the following theorem due to Baouendi, Ebenfelt and Rothschild   (Theorem 12.3.11 and Remark 12.3.13 in \cite{BER}) for the model quadric $M_H$.
\begin{theo}\label{egypte}\cite{BER}
Let $M_H \subset \C^N$ be a model  quadric   given by  \eqref{eqred1}. Assume  that $M_H$ is Levi nondegenerate at $0$
and that there exists $V\in \C^n$ such that  the  $n \times d$ matrix  $D_0$ whose $j^{th}$ column is
 $A_jV$ is of rank $d.$ Then any  biholomorphism sending $M_H$ into itself is uniquely determined by its $2$-jet at $0.$
\end{theo}
\begin{proof}
This is a direct application of Lemma \ref{nil0} and Proposition 11.1.12 in \cite{BER}.
\end{proof}  
Using  Theorem \ref{egypte},  we then obtain the following new result  which is a direct application of  Theorem 3.10 in \cite{bl-me1} for   smooth  generic submanifolds:
\begin{theo}\label{nil1}
Let $M \subset \C^N$ be a  smooth  submanifold given by  \eqref{eqred0}. Assume that $M$ is Levi nondegenerate at $0$ and that  the matrix  $D_0$ whose $j^{th}$ column is
 $A_jV$ is of rank $d.$ Then any  biholomorphism sending $M$ into itself is uniquely determined by its $2$-jet at $0.$
\end{theo}
\begin{proof}
 Let  $hol(M_H,0)$  be the set of  germs of real-analytic 
infinitesimal CR automorphisms of $M_H$ at $0.$ By Theorem \ref{egypte}, the elements of $hol(M_H,0)$ are  determined by their $2$-jets.
Hence, a direct application of Theorem 3.10 in \cite{bl-me1} yields to the result. 
\end{proof}
 
\begin{remark} The second author recently wrote a paper with  Kol\'a\v r \cite{ko-me} in which they observe that it is enough to assume that $M$ is of class $\mathcal{C}^3$ in the previous theorem. 
\end{remark}
   
\section{Stationary minimal submanifolds}
Let $M$ be a 
   strongly Levi nondegenerate real submanifold given by \eqref{eqred0}. We fix $b\in \R^d$   such that $\sum_{j=1}^d b_jA_j$  is invertible, and let $a \in \C^d$ be sufficiently small.  We  consider the  following  quadratic matrix equation 
\begin{equation}\label{confi00}
PX^2 + AX +   \transp{\overline {P}}=0
\end{equation} 
where $P:=\sum_{j=1}^d {{a}_j} A_j $ and   $A:=\sum_{j=1}^d (b_j -a_j-\overline{a_j})A_j.$  Consider  $X$ to be the unique $n\times n$ matrix solution 
  of \eqref{confi00} such that $\|X\|<1.$  Note that $X$ depends on both $a$ and $b$. The following definition, inspired by the work of Tumanov in the strongly pseudoconvex setting \cite{tu}, was introduced in \cite{be-me2}.
 \begin{defi}\label{defstatmin}
 Let $M$ be a 
   strongly Levi nondegenerate (at $0$) real submanifold given by \eqref{eqred0}. Let $b\in \R^d$ be such that $\sum_{j=1}^d b_jA_j$ is invertible and let  $V\in \C^n$. Consider $a \in \C^d$ sufficiently small and the  solution $X$  of \eqref{confi00} with $\|X\|<1$. We say that $M$ is {\it stationary minimal at $0$ for $(a,b -a-\overline{a},V)$} if the matrices 
 $A_1,\ldots,A_d$ restricted to the orbit space
 $$\mathcal{O}_{X, V}:={\rm span}_\R \{V, XV, {X^2}V, \ldots, {X^k}V, \ldots\}$$ are 
 $\R$-linearly independent. 
 \end{defi}
Note that the above definition is independent of the choice of holomorphic coordinates and that
  $M$ is stationary minimal for $(0,b,V)$ in case it is $\mathfrak{D}$-nondegenerate.
Tumanov proved that if  $A_1,\ldots,A_d$ are $d$ linearly independent Hermitian matrices, then there exist $a\in \Bbb R^d $ and $V \in \C^n$ such that  the matrices 
 $A_1,\ldots,A_d$ restricted to the  space
$${\rm span}_\R \{V, PV, {P^2}V, \ldots, {P^k}V, \ldots\}$$ are 
 $\R$-linearly independent (Theorem 5.3 in \cite{tu3}). In particular, this implies that strongly pseudoconvex Levi generating submanifolds are stationary minimal. 
 For the sake of clarity, we   provide a proof. 
\begin{theo} \label{cru}\cite{tu3}
Let $M \subset \C^N$ be a  $\mathcal{C}^4$ generic  real submanifold  given by \eqref{eqred0}.  Assume that $M$ is strongly pseudoconvex and that the matrices 
$A_1,\ldots,A_d$ are linearly independent. Then  there exists  $a,b\in \Bbb R^d$ and $V \in \C^n$ 
 such that $M$ is stationary minimal at $0$ for $(a,b -2a,V).$ 
\end{theo}
\begin{proof}  Let $b\in \R^d$ be  such that $\sum_{j=1}^d b_jA_j$  is positive definite; without loss of generality, we  assume that $\sum_{j=1}^d b_jA_j$ is the identity matrix. 
According to  Theorem  5.3 in \cite{tu3}, there exists $a_0 \in \Bbb R^d$  and $V \in \Bbb C^n$ such that the  matrices $A_1,\ldots,A_d$ are  $\R$-linearly independent on the space 
\begin{equation*}
\mathcal{O}_{(a_0, V)}:={\rm span}_\R \{V, PV, {{P}^2}V, \dots, {{P}^{s}}V, \dots \}, 
\end{equation*}
 with $P=\sum {a_0}_jA_j$. Note that since $a_0$ can be chosen as small as necessary, the matrix $A:=\sum_{j=1}^d (b_j -2{a_0}_j)A_j$ is positive definite.  
Taking $s$ large enough,  we assume that 
 $$\mathcal{O}_{(a_0, V)}={\rm span}_\R \{V, PV, P^2V, \dots, P^{s}V \}.$$
We set 
 $$\tilde V=(V, PV, P^2V, \dots, P^{s}V) \in \C^{(s+1)n},$$ 
and we define  
$\tilde {A_j}, j=1\ldots,d,$  and $\tilde {P}$ to  be  the following $(s+1)n \times (s+1)n$ matrices
$$
\tilde {A_j}:=\begin{pmatrix}
	A_j& & & (0) \\ &A_j & & \\ & & \ddots & \\ (0)& & &A_j
\end{pmatrix}, \ \ 
\tilde {P}:=\begin{pmatrix}
	P & & & (0) \\ &P & & \\ & & \ddots & \\ (0)& & &P
\end{pmatrix}.
$$
It follows that  the vectors 
 $\tilde {A}_1{\tilde V},  \dots, \tilde {A}_d{\tilde V}$
 are $\R$-linearly independent.  We may then find an invertible  $d\times d$ matrix  $C_{hom}$ whose  $j^{th}$ column is  given by well chosen  (independent of $j$) components of  $\tilde {A}_j{\tilde V}$. Let  $X$ be the matrix solution of \eqref{confi00} with $\|X\|<1$. We replace ${\tilde V}$ by $\tilde V_X,$  where  $$\tilde V_X=(V, XV, {X}^2V, \dots, {X}^{s}V) \in \C^{(s+1)n},$$  with  $X$ evaluated at $a_0,$
 and we consider the $d\times d$ matrix $C$ obtained from 
 $\tilde {A}_1{\tilde V_X},  \dots, \tilde {A}_d{\tilde V_X}$ by keeping the same rows as the ones of $C_{hom}$.   Since $X$ is of the form
$$X=-\left(I-2\sum_{j=1}^d {a_0}_jA_j\right)^{-1} \transp{\overline {P}} + O (|a_0|^3)=- \transp{\overline {P}} + O (|a_0|^2),$$  
it follows that $C$ has a determinant with a first  nonzero  homogeneous term of degree $d$ depending on  $a_0,$ actually the determinant of  $C_{hom}.$ Therefore, expanding the determinant of $C$ into homogeneous terms, and replacing $a_0$ by $\lambda a_0,$ we obtain a function $h(\lambda)$  with respect to $\lambda$    whose first term coefficient  of order $d$ is non zero. Therefore,  we  may choose $\lambda_0$ such that $h(\lambda_0)\ne 0. $ This shows that $M$ is stationary minimal for $a=\lambda_0 a_0.$

\end{proof}

 \section{Stationary discs}

Let $M\subset \C^N$ be a $\mathcal{C}^{4}$ generic  real submanifold of codimension $d$ given by  (\ref{eqred0}). Following Lempert \cite{le} and Tumanov \cite{tu},  a holomorphic disc $f: \Delta \to \C^N$ continuous up to  $\partial \Delta$ and such that $f(\partial \Delta) \subset M$ is {\it stationary for $M$} if there 
exist  $d$ real valued functions $c_1, \ldots, c_d : \partial \Delta \to \R$ such that $\sum_{j=1}^dc_j(\zeta)\partial r_j(0)\neq 0$ for all $\zeta \in \partial \Delta$   and such that the map denoted by $\tilde{f}$
\begin{equation*}
\zeta \mapsto \zeta \sum_{j=1}^dc_j(\zeta)\partial r_j\left(f(\zeta), \overline{f(\zeta)}\right)
\end{equation*}
defined on $\partial \Delta$ extends holomorphically on $\Delta$. In that case, the disc $\bm{f}=(f,\tilde{f})$ is a holomorphic lift of $f$ to the cotangent bundle $T^*\C^{N}$ and the set of all such lifts $\bm{f}=(f,\tilde{f})$, with $f$ nonconstant, is denoted by $\mathcal{S}(M)$. 

We now consider a strongly Levi nondegenerate quadric $M_H$ given by \eqref{eqred1}. We fix $b \in  \Bbb R^d$  such that $\sum_{j=1}^d {b}_jA_j$ is invertible and
 we define $\mathcal{S}_0(M_H) \subset \mathcal{S}(M_H)$ to be the subset of lifts whose value at $\zeta=1$ is $(0,0,0,b/2)$.  
 Consider an initial  disc $\bm {f_0} \in \mathcal{S}_0(M_H)$ given by
 \begin{equation*}
 \bm {f_0}=\left((1-\zeta)V_0,2(1-\zeta)^t \overline{V_0}A_1V_0,\ldots,2(1-\zeta)^t \overline{V}A_dV_0,(1-\zeta) ^t \overline{V_0}(\sum {b}_jA_j), \frac{\zeta}{2} b\right),
 \end{equation*} 
with $V_0 \in\C^n$. In \cite{be-me2}, we obtained the explicit expression of lifts of stationary discs near $\bm {f_0}$. Actually, that result is due to Tumanov \cite{tu} when the quadric $M_H$ is strongly pseudoconvex. More precisely lifts of stationary discs $\bm{f}=(h,g,\tilde{h},\tilde{g}) \in \mathcal{S}_0(M_H)$   near  $\bm {f_0}$ are of the form
\begin{equation}\label{eqform}
\begin{cases}
h(\zeta)= V-\zeta(I-\zeta X)^{-1}(I-X)V \\
\\
g_j(\zeta)={^t\overline{V}}A_j V-2^t \overline {V}A_j \zeta (I-\zeta X)^{-1}(I-X)V +^t \overline {V}(^t\overline {X}K_j-K_jX))V +\\
\hspace{1.4cm}^t \overline {V}(I-^t \overline {X})K_j(I + 2\zeta X(I-\zeta X)^{-1})(I-X) V, \\
\\
\tilde{h}(\zeta)=-\zeta\transp\overline{h(\zeta)} \left({\sum_{j=1}^d ({a_j}\overline{\zeta}+(b_j-a-\overline a)+\overline {a_j}\zeta)A_j}\right)\\
\\
\displaystyle \tilde{g}(\zeta)=\frac{a+(b-a-\overline a)\zeta+\overline{a}\zeta^2}{2}\\
\end{cases}
\end{equation}
where  $V \in \Bbb C^n$ (close to $V_0$), $a \in \C^d$ is sufficiently small, $X$ is the unique $n\times n$ matrix solution of \eqref{confi00} (with $b- a-\overline{a}$) with $\|X\|<1$, and for $j=1, \dots,d$ 
\begin{equation}\label{eqK_j}
K_j=\sum_{r=0}^\infty {\transp \overline {X}}^r A_j X^r. 
\end{equation}
\vspace{0.5cm}

As a consequence of that explicit expression, it is possible to characterize nondefective discs in $\mathcal{S}_0(M_H)$. 
Recall from \cite{ba-ro-tr} that a stationary disc $f$  is {\it defective} if it admits a lift 
$\bm{f}=(f,\tilde{f}): \Delta \to T^*\C^N$  such that $(f,\tilde{f}/\zeta)$ is holomorphic 
on $\Delta$. As proved in  \cite{be-me2}, a disc $f$ of the form \eqref{eqform}  is nondefective if and only if  the quadric $M_H$ 
is stationary minimal at $0$ for $(a,b-a-\overline{a},h(0))$. Together with Theorem \ref{cru}, we then recover
\begin{theo} \label{theoexist}\cite{tu3}
Let $M \subset \C^N$ be a  $\mathcal{C}^4$ generic  real submanifold  given by \eqref{eqred0}.  Assume that $M$ is strongly pseudoconvex and that the matrices 
$A_1,\ldots,A_d$ are linearly independent.  Then  there exists a nondefective stationary disc $f$ of the form \eqref{eqform}.
\end{theo}

\section{ The $1$-jet  map for stationary discs of the model quadric}
Let $M_H$ be a model quadric of the form \eqref{eqred1}. 
 Consider the $1$-jet map at $\zeta=1$ defined on $\mathcal{S}_0(M_H)$ 
 $$\mathfrak j_{1}:\bm{f} \mapsto (\bm{f}(1),\bm{f}'(1)).$$ Since $\bm{f}(1)=(0,0,0,b/2)$ where $b \in \R^d$ is fixed, the  $1$-jet map is identified with the derivative map  
 $\bm{f} \mapsto \bm{f}'(1)$ at $\zeta=1$. Using the explicit expression of lifts \eqref{eqform}, we note in \cite{be-me2} that,  after changes of variables in both the source and the target spaces,                                                                        
 the  $1$-jet map   $\mathfrak j_{1}: \C^d \times \C^n \to  \C^n \times \R^d \times \C^d$ at $\zeta=1$ may be written as 
\begin{equation}\label{rikip} 
\mathfrak j_{1}
: (a,V) \mapsto 
(V, \transp\overline{V}(I-\transp\overline {X})K_j(I-X)V, \Im m {a}).
\end{equation} 
We prove in  \cite{be-me2} that  the fact that  $M_H$  is stationary minimal at $0$  is a necessary condition for the $1$-jet map $\mathfrak j_{1}$ to be a local diffeomorphism. In this section, we  give an elementary proof of the 
equivalence  when the   model quadric $M_H$ is strongly pseudoconvex.

  \vspace{0.5cm} 
  
Denote by  $\mathcal{M}_n(\C)$ the space of square matrices of size $n$ with complex coefficients.  Let $P$ and $X$ be  given by \eqref{confi00}. We consider the map 
 \begin{equation*}
 \varphi: \mathcal{M}_n(\C) \to \mathcal{M}_n(\C),\  \ 
\end{equation*}
defined by 
\begin{equation}\label{symph} 
\varphi(N)=N+P(NX+XN).  
\end{equation}
For  $a \in \C^d$  small enough, the map $\varphi$ is  invertible with inverse  of the form
\begin{equation*}
\varphi^{-1}(N)= N +\sum_{r=0}^\infty Q_r(P,X)N X^r,
\end{equation*} 
where   $Q_r$ is a convergent power series in $P$ and $X$; e.g.  $Q_0=\sum_{k \ge 1} (-1)^k(PX)^k$. Moreover $Q_r (P,X)=\sum_{k\ge 1} R_{k,r}(P,X)$ where $R_{k,r}(P,X)$ is a homogeneous polynomial of degree $k$ in $P$ and $k-r$ in $X$.
Note that if $a=0$ then  $\varphi $ is the identity map. In what follows, we denote by $X_{\Re e a_s}$ the derivative $\dfrac{\partial X}{\partial {\Re e a_s}}$, $s=1,\ldots,d$. 
We have the following lemma. 
\begin{lemma}\label{klop1} 
Let $M_H$ be a strongly  pseudoconvex  quadric given by \eqref{eqred1}.  Let $a \in \R^d$ be small enough and let  X be the unique $n\times n$ matrix solution of \eqref{confi00}  such that $\|X\|<1$.  
Then, after a linear choice of coordinates, we have
\begin{equation*}
 X_{ \Re e a_s}= \varphi^{-1}(-A_s (I-X)^2)=-\varphi^{-1}(A_s) (I-X)^2,
\end{equation*}
where  $\varphi$ is given by \eqref{symph}.
\end{lemma}

\begin{proof}
 Since $M_H$ is strongly pseudoconvex,   we  choose coordinates for which $\sum_{j=1}^d ({b}_j-{a}_j-\overline{a_j})A_j =I.$ Differentiating  Equation 
\eqref{confi00} in $\Re e { a_s}$ and evaluating at $(a,b- a-\overline{a})$ implies that
$$A_s X^2 + P (XX_{\Re e a_s}+X_{\Re e a_s}X)-2A_sX+X_{\Re e a_s}+A_s=0$$
and so
$$X_{\Re e a_s}+P (XX_{\Re e a_s}+X_{\Re e a_s}X) = -A_s(I-X)^2$$
We then obtain
 \begin{equation*}
 X_{\Re e a_s} = \varphi^{-1}(-A_s (I-X)^2)=-\varphi^{-1}(A_s) (I-X)^2.
 \end{equation*}  
\end{proof}
We recall the following lemma.
\begin{lemma} \cite{be-me2}\label{tortue1}  Let $M_H$ be a strongly Levi nondegenerate quadric given by \eqref{eqred1} and let $b \in \R^d$ be such that $\sum_{j=1}^d b_jA_j$ is invertible.
 Let $a \in \C^d$ be small enough and let  X be the unique $n\times n$ matrix solution of \eqref{confi00} such that $\|X\|<1$.  
 \begin{enumerate}[i.]
\item  The $1$-jet map $\mathfrak j_{1}$ is a local diffeomorphism at $(a,V) \in \C^d\times \C^n$ if and only if 
the  $d\times d$ matrix 
$$\left(\dfrac{\partial}{\partial {\Re e a_s}}\transp\overline{V}(I-\transp\overline {X})K_j(I-X)V\right)_{j,s}$$ is invertible.
\item
For any $s=1,\ldots,d$, we have 
\begin{eqnarray*}
\dfrac{\partial}{\partial {\Re e a_s}}\transp\overline{V}(I-\transp\overline {X})K_j(I-X)V& = & -2\Re e \left(\sum_{r=0}^{\infty} \transp\overline{V}\left(I-\transp \overline {X}\right)^2 {\transp \overline {X}}^rK_jX_{\Re e a_s}X^rV\right).\\
\end{eqnarray*}
\end{enumerate}
\end{lemma}
 
The main theorem is
\begin{theo}\label{propder}
Let $M_H$ be a strongly pseudoconvex quadric given by \eqref{eqred1} and let $b \in \R^d$ be such that $\sum_{j=1}^d b_jA_j$ is invertible.
Assume that $M_H$ is stationary minimal at $0$ for $(a,b-a-\overline{a},V)$ with $a \in \C^d $ sufficiently small.
Then the  $1$-jet map $\mathfrak j_{1}$ given by  \eqref{rikip} is a local diffeomorphism at $(a,V). $
\end{theo}

\begin{proof} 
Assume that $M_H$ is stationary minimal at $0$ for $(a,b-a-\overline{a},V)$ with $a \in \C^d $ sufficiently small. Due to the strong pseudoconvexity of $M_H$, we  assume that the 
matrix $A:=\sum_{j=1}^d (b_j -a_j-\overline{a_j})A_j$ is the identity.
 In light of  Lemma \ref{tortue1}, we need to show that 
$$Y:=\Re e\left(\sum_{r=0}^{\infty} \transp\overline{V} \left(I-\transp \overline {X}\right)^2 {\transp \overline {X}}^rK_jX_{\Re e a_s}X^r V\right)_{j,s}$$
  is invertible. In fact, we will prove it is negative definite. Using  Lemma \ref{klop1}, we have:
\begin{eqnarray*}
Y& =   &- \Re e\left(\sum_{r=0}^{\infty} \transp\overline{V} \left(I-\transp \overline {X}\right)^2 {\transp \overline {X}}^rK_j\varphi^{-1}(A_s)X^r (I-X^2)V\right)_{j,s}  \\ 
& = &- \Re e\left(\sum_{r=0}^{\infty} \transp\overline{V'} {\transp \overline {X}}^rK_j\varphi^{-1}(A_s)X^r V'\right)_{j,s}  \\ 
\end{eqnarray*}
where $V'= (I-X)^2V$. Since the quadric $M_H$ is stationary minimal at $0$ for $(a,b-a-\overline{a},V)$,  it is also  stationary minimal at 
$0$ for $(a,b-a-\overline{a},(I-X)^2V)$ (see the proof of Lemma 6.7 \cite{tu}, see also \cite{be-me2}.)
  We then  abusively write $V$ in place of $V'$, and we need to show that the matrix $\Re e\left(\sum_{r=0}^{\infty} \transp\overline{V} {\transp \overline {X}}^rK_j\varphi^{-1}(A_s)X^r V\right)_{j,s}$ is positive definite.
	
Let $W \in\C^d$ be a nonzero unit vector. Since $M$ is stationary minimal, we have 
$$\transp \overline{W}\Re e \left(\sum_{r=0}^{\infty} \transp \overline V{\transp \overline {X}}^rA_jA_sX^r V\right)_{j,s}W = \sum_{r=0}^\infty \|D_r W\|^2 + \sum_{r=0}^\infty \|\overline{D_r} W\|^2> 0,$$ 
where $D_r$ is the $n \times d$ matrix whose $s^{th}$ column is $A_sX^r V$. Since $X$ is solution of Equation \eqref{confi00}, the term $D_rW$ is a sum of homogeneous polynomials in $a,\overline{a}$ of degree greater than or equal to $r$ (and congruent to $r$ modulo $2$). Denote by $k(r)$ the minimal degree appearing in this decomposition and define $k_0:=\min_{r\geq 0} k(r).$ We denote by $I$ the finite set of integers $ r$ which realize the mininum, that is 
 such that  $D_rW$ contains terms of degree $k_0$. Note that if $r \notin I$ then $D_rW=O(\|a\|^{k_0+1})$. Note also that the minimal degree appearing $\sum_{r=0}^\infty \|D_r W\|^2$ is exactly $2k_0$.  
Moreover, the function $k_0$ is a upper semi-continuous  in $W$, which by compactness implies that it is bounded by above; this fact is crucial since it allows to choose $a$ smaller if necessary. 
Now define
$$S_r:= \transp \overline W\left(\transp\overline{V}  {\transp \overline {X}}^rK_j\varphi^{-1}(A_s)X^r V\right)_{j,s}W.$$ 
In case $r\in I$ then 
$$S_r=\|D_rW\|^2+O(\|a\|^{2k_0+2})>0.$$ We claim that if $r\notin I$ then $S_r=O(\|a\|^{2k_0+2})$. To show this, note that, in that case,  $S_r$ only contains terms of the form 
$$\transp \overline W\left(\transp\overline{V}  {\transp \overline {X}}^{r+\ell_1}A_jX^{\ell_1} R_{k,\ell_2}(P,X)A_sX^{r+\ell_2} V\right)_{j,s}W=\transp \overline{D_{r+\ell_1}W}X^{\ell_1}R_{k,\ell_2}(P,X)D_{r+\ell_2}W,$$
where $R_{k,\ell_2}(P,X)$ is a homogeneous polynomial of degree $k$ in $P$ and $k-\ell_2$ in $X$; this is due to the form of the inverse of $\varphi$ (see Equation \eqref{symph}).
Whether $r+\ell_1$ and $r+\ell_2$ are in $I$ or not, we have 
$$\transp \overline{D_{r+\ell_1}W}X^{\ell_1}R_{k,\ell_2}(P,X)D_{r+\ell_2}W=O(\|a\|^{2k_0+2}).$$ Indeed, we distinguish the four following cases.
\begin{itemize}
\item  If $r+\ell_1 \notin I$ and $r+\ell_2 \notin I$, this is clear since $D_{r+\ell_j}W=O(\|a\|^{k_0+1})$, $j=1,2$.

\item The case $r+\ell_1 \notin I$ and $r+\ell_2 \in I$ can only occur if $k\geq 1$.  In that case, we have $D_{r+\ell_1}W=O(\|a\|^{k_0+1})$ and $D_{r+\ell_2}W=O(\|a\|^{k_0})$. Due to the contribution of  the term $R_{k,\ell_2}(P,X)$, we obtain  $O(\|a\|^{2k_0+2})$.  

\item The case $r+\ell_1 \in I$ and $r+\ell_2 \notin I$ can only occur if $\ell_1\geq 1$.  We have $D_{r+\ell_1}W=O(\|a\|^{k_0})$, $D_{r+\ell_2}W=O(\|a\|^{k_0+1})$ and with the contribution of the term $X^{\ell_1}$ we obtain  $O(\|a\|^{2k_0+2})$.

\item Finally, the case $r+\ell_1 \in I$ and $r+\ell_2 \in I$ can only occur if $\ell_1\geq 1$ and $k\geq 1$.  We have $D_{r+\ell_j}W=O(\|a\|^{k_0})$, $j=1,2$,
and  with the contribution of the terms $X^{\ell_1}$  and $R_{k,\ell_2}(P,X)$ we also obtain  $O(\|a\|^{2k_0+2})$.
\end{itemize}
At this stage, we have proved that
$$\sum_{r=0}^{\infty} S_r=\sum_{r \in I}\|D_rW\|^2+O(\|a\|^{2k_0+2})$$
 and so is positive since $a$ can be taken smaller if necessary. This proves that $Y$ is negative definite and concludes the proof of the theorem.

\end{proof}

The previous theorem and Theorem \ref{cru} imply the following result.
\begin{cor} Let $M_H$ be a strongly pseudoconvex quadric given by \eqref{eqred1}.  Assume that the matrices 
$A_1,\ldots,A_d$ are linearly independent. Then there exist $a \in \R^d$ and $V \in \C^n$ such that the  $1$-jet map $\mathfrak j_{1}$ is a local diffeomorphism at $(a,V).$
\end{cor}
As a direct consequence, we recover  Theorem \ref{strictly} by the usual stationary disc method \cite{be-bl, be-bl-me, be-me}. 

\subsection{Open questions}
An interesting feature of Theorem \ref{propder} is the fact that its proof gives a roadmap for extending the result to quadrics and submanifolds which are not necessarily strongly pseudoconvex. In this vein, we define
\begin{defi}\label{definondeg100}
  Let $M$ be a real submanifold given by \eqref{eqred0}. Assume $M$ is strongly Levi nondegenerate at $0$ and let $b\in \R^d$ be such that $\sum_{j=1}^d b_jA_j$ is invertible. Let $a \in \C^d$ be sufficiently small and set  $A:=\sum_{j=1}^d (b_j -a_j-\overline{a_j})A_j.$  
 We say that  $M$ is (resp. {\it strongly}) {\it  $\mathfrak{D}(a)$-nondegenerate} at $0$  if there exists $V\in \C^n$ such that  the  matrix 
\begin {equation*}
\Re e \left(\sum_{r=0}^{\infty} \transp \overline V{\transp \overline {X}}^rA_j A^{-1}A_sX^r V\right)_{j,s}
\end {equation*}
is nondegenerate (resp. positive definite). 
\end{defi}

Note that when $a = 0$, we recover the definition of  $\mathfrak{D}$-nondegeneracy.
The following lemma is immediate.
\begin{lemma} Definition \ref{definondeg100} is  independent of the choice of holomorphic coordinates. 
\end{lemma}
We also have
\begin{lemma} If  $M$ is  {\it  $\mathfrak{D}(a)$-nondegenerate} at $0$  then $M$ is stationary minimal at $0$ for $(a,b -a-\overline a,V)$ for some $V\in \C^n$.  
\end{lemma}
\begin{proof}
Assume that $\lambda_1,\ldots,\lambda_d \in \R$ are such that  $\sum_{j=1}^d\lambda_jA_jX^rV=0$ for all $r=0,1,2,\ldots$. Set $W=\transp(\lambda_1,\ldots,\lambda_d) \in \R^d$ and we consider the $n \times d$ matrix $D_r$ whose $s^{th}$ column is $A_sX^r V$. We then have $D_rW=0$ for all $r=0,1,2,\ldots$. It follows that

\begin{eqnarray*}
\Re e \left(\sum_{r=0}^{\infty} \transp \overline V{\transp \overline {X}}^rA_j A^{-1}A_sX^r V\right)_{j,s}W& =   &  \transp{\overline{D_r}}A^{-1}D_rW +   \transp{D_r}\overline{A}^{-1}\overline{D_r}W 
 \\ 
& = &  \transp{D_r}\transp{A}^{-1}\overline{D_r W} =0\\ 
\end{eqnarray*}
Since $M$ is $\mathfrak{D}(a)$-nondegenerate, it follows that $W=0$.   
\end{proof}
We also note that in case $M$ is strongly pseudoconvex then if $M$  is stationary minimal at $0$ for $(a,b -a-\overline a,V)$ then it is  (strongly) $\mathfrak{D}(a)$-nondegenerate at $0$; in fact, according to Theorem \ref{cru}, if $M$ is strongly pseudoconvex and Levi generating then it is (strongly) $\mathfrak{D}(a)$-nondegenerate for some $a \in \C^d$.

In  \cite{be-me}, we conjectured that if $M$  is a $\mathcal{C}^4$ generic  real submanifold strongly Levi nondegenerate of the form  \eqref{eqred0}, and admitting a nondefective stationary disc passing through  $0$, then  any germ  at $0$ of  CR automorphism  of $M$ of class $\mathcal{C}^3$ is  uniquely determined by its $2$-jet at $p.$ The key point in the conjecture is to show that the $1$-jet map $\mathfrak j_{1}$ is a local diffeomorphism. In \cite{be-me}, we proved that if $M_H$ is strongly Levi nondegenerate quadric and the map $\mathfrak j_{1}$ is a local diffeomorphism, then $M_H$ is stationary minimal. At the moment, it is not clear to us how to obtain the converse.

\vspace{0.3cm}
\noindent {\bf Questions.} {\it Consider a $\mathcal{C}^4$ generic  real submanifold $M \subset \C^{n+d}$. Assume that $M$ is (strongly) $\mathfrak{D}(a)$-nondegenerate at $0$.}
\begin{enumerate}[i.]
\item {\it Is the $1$-jet map $\mathfrak j_{1}$  a local diffeomorphism at $(a,V)$ for some $V \in \C^n$?}
 \item {\it Are germs  at $0$ of biholomorphisms sending $M$ into itself uniquely determined by their $2$-jet at $0$?}    
 \item {\it Are germs  at $0$ of  CR automorphisms  of $M$ of class $\mathcal{C}^3$  uniquely determined by their $2$-jet at $0$?}    
 \end{enumerate}
\vspace{0.3cm}
 Note that although i. implies ii. and iii., it may be possible to prove the second and third points  ii. and iii. without using the stationary disc method via different techniques.

\vskip 1cm
{\small
\noindent Florian Bertrand\\
Department of Mathematics,\\
American University of Beirut, Beirut, Lebanon\\
{\sl E-mail address}: fb31@aub.edu.lb\\

\noindent Francine Meylan \\
Department of Mathematics\\
University of Fribourg, CH 1700 Perolles, Fribourg\\
{\sl E-mail address}: francine.meylan@unifr.ch\\
} 


\begin{thebibliography}{11111} 


\bibitem{BER}M.S. Baouendi, P. Ebenfelt, L.P. Rothschild, {\it Real submanifolds in complex space and their mappings}, 
Princeton Mathematical Series, {\bf 47}. Princeton University Press, Princeton, NJ, 1999. xii+404 pp.


  
\bibitem{ba-ro-tr}  M.S. Baouendi, L.P. Rothschild, J.-M. Tr\' epreau,   {\it On the geometry of analytic discs attached to real manifolds}, 
J. Differential Geom.  \textbf{39} (1994),   379-405.  


\bibitem{be-bl} F. Bertrand, L. Blanc-Centi, {\it Stationary holomorphic discs and finite jet determination problems},  
Math. Ann. {\bf 358} (2014), 477-509.

\bibitem{be-bl-me} F. Bertrand, L. Blanc-Centi, F. Meylan, {\it Stationary discs and finite jet determination for non-degenerate generic real submanifolds},
Adv. Math. {\bf 343} (2019), 910-934. 


\bibitem{be-me} F. Bertrand, F. Meylan, {\it Nondefective stationary discs and $2$-jet  determination in higher codimension},  J. Geom. Anal. {\it 31} (2021), 6292-6306.




\bibitem{be-me2} F. Bertrand, F. Meylan, {\it Explicit construction of stationary discs and its
consequences for nondegenerate quadrics}, preprint, arXiv:2111.10201.

\bibitem{bl-me}  L. Blanc-Centi, F. Meylan, {\it On  nondegeneracy conditions  for the Levi map in higher codimension: a Survey},
Complex Anal. Synerg. {\bf 6} (2020). 

\bibitem{bl-me1}  L. Blanc-Centi, F. Meylan, {\it Chern-Moser operators and weighted jet determination problems in higher codimension}, preprint, arXiv:1712.00295. 


\bibitem{gr-me} J. Gregorovi\v c, F. Meylan, {\it  Construction of counterexamples to the $2$-jet determination Chern-Moser Theorem in higher codimension}, to appear in  Math. Res. Lett.
 

 
\bibitem{ko-me} M. Kol\'a\v r, F.  Meylan, {\it Remarks on the symmetries of a model hypersurface}, prepring (2021). 
  
  \bibitem{la-mi3} B. Lamel,  N. Mir, {\it Two decades of finite jet determination of CR mappings}, preprint (2021).
   
 
  \bibitem{le} L. Lempert, {\it La m\'etrique de Kobayashi et la repr\'esentation des domaines sur la boule}, Bull. Soc. Math. France 
{\bf 109} (1981), 427-474.





\bibitem{sc-tu} A. Scalari, A. Tumanov,  {\it Extremal discs and analytic continuation of product CR maps}, Michigan Math. J. {\bf 55} (2007), 25-33.





\bibitem{tu} A. Tumanov, {\it Extremal discs and the regularity of CR mappings in higher codimension}, Amer. J. Math. 
{\bf 123} (2001), 445-473.

\bibitem{tu3} A. Tumanov, {\it Stationary Discs and finite jet determination for CR mappings in higher codimension},   Adv. Math. {\bf 371} (2020), 107254, 11 pp.

\bibitem{za} D. Zaitsev, {\it Germs of local automorphisms of real analytic CR structures and analytic dependence on the $k$-jets}, 
Math. Res. Lett. {\bf 4} (1997), 1-20.

\end{thebibliography}
\end{document}